\documentclass[preprint,12pt]{elsarticle}

\usepackage{amssymb}
\usepackage{amsmath}
\usepackage{amsthm}
\usepackage{url}

\newtheorem{prop}{Proposition}
\journal{Involve}

\begin{document}

\begin{frontmatter}

\title{Exploring Integration by Differentiation} 

\author{R. D. George and C. Vignat}

\affiliation{organization={Department of Mathematics},
            addressline={Tulane University}, 
            city={New Orleans},
            country={U.S.A.}}
\ead{rgeorge7@tulane.edu, cvignat@tulane.edu}
\begin{abstract}
This work validates and extends the method of integration by differentiation, initially introduced by A. Kempf et al., and demonstrates its compatibility with classical rules of integration. It provides applications to classical integrals, including one by Ramanujan, and extends the method to the multivariate setting. Volumes of simplexes are computed by acting with indicator functions on elementary kernels, and a rotationally invariant formulation is derived. Finally, the method is extended to Jackson's q-integral.
\end{abstract}


\end{frontmatter}

%\tableofcontents

%\section{Introduction}
\section{The Method of Integration by Differentiation}
\subsection{Origins and Motivation}
The method of integration by differentiation was introduced by Kempf, Jackson, and Morales~\cite{Kempf, Kempfb} and further developed by Jia, Tang, and Kempf~\cite{Jia}. Initially motivated by problems in quantum field theory, it exploits the relationship between differentiation and multiplication inside integrals of the Laplace transform type.

\subsection{Core Idea}
The function $x\mapsto e^{-xy}$ is an eigenfunction of the differentiation operator \(\partial_x = \frac{d}{dx}\), satisfying
\[
\partial_x^n e^{-xy} = (-y)^n e^{-xy}.
\]
Thus, for any polynomial and, by extension, for functions $f$ satisfying the appropriate convergence conditions established in \cite{Jia},
\begin{equation}
\label{main property}
f(-\partial_x) e^{-xy} = f(y) e^{-xy}.    
\end{equation}
One example of such analytic function for which \eqref{main property} holds is the exponential function $f(x)=e^{ax},$  
as a consequence of the shift property $e^{-a \partial_x}g(x)=g(x-a)$ of the exponential differential operator $e^{-a \partial_x}.$

Using property \eqref{main property},   Laplace-type integrals can be rewritten as
% of the form $\int_0^\infty f(y) e^{-xy} \, dy$ 
as
\begin{align*}
    \int_0^\infty f(y) e^{-xy} \, dy 
    &= \int_0^\infty f(-\partial_x) e^{-xy} \, dy. 
\end{align*}
Noticing that the operator $f(-\partial_x)$ does not depend on the integration variable $y,$ we deduce the formal identity
\[
\int_0^\infty f(y) e^{-xy} \, dy
= f(-\partial_x) \int_0^\infty e^{-xy} \, dy 
    = f(-\partial_x) \left( \frac{1}{x} \right).
\]
Hence the integral can be computed as the result of the action of the pseudo-differential operator \(f(-\partial_x)\) on the function $\frac{1}{x}$ through the power series expansion of \(f\).
For more information on the concept of pseudo-differential operators, the reader is invited to consult \cite{Stein}.

\subsection{General Formulation}
More generally, for an arbitrary integration domain \(I\) where the integral is convergent,
\[
\int_I f(y) e^{-xy} \, dy = f(-\partial_x)  \int_I e^{-xy} \, dy.
\]
For example,
\[
\int_a^b f(y) e^{-xy} \, dy = f(-\partial_x)  \frac{e^{-ax} - e^{-bx}}{x},
\]
and
\[
\int_{-\infty}^{\infty} f(y) e^{-\imath xy} \, dy = 2\pi f(-\imath \partial_x) \delta(x),
\]
where $\delta(x)= \int_{-\infty}^{\infty} e^{-\imath xy} \, dy$
denotes the Dirac distribution.

Moreover, whenever the limit and integration operators can be exchanged, the specialization
\[
\int_I f(y)  \, dy = \lim_{x \to 0}
 \int_I f(y) e^{-xy} \, dy
\]
produces the following rules
\begin{equation}
\label{Laplace limit}
\int_{0}^{\infty} f(y)  \, dy
=
\lim_{x \to 0}
f(\partial_x)\left(\frac{1}{x}\right),
\end{equation}
\begin{equation}
\label{Fourier limit}
\int_{-\infty}^{\infty} f(y)  \, dy
=
2\pi\lim_{x \to 0}
 f(-\imath \partial_x)\delta(x).
\end{equation}

\subsection{Convergence and Validity Conditions}
Following Jia, Tang, and Kempf~\cite{Jia}, the integration by differentiation method requires specific conditions:
\begin{itemize}
\item For finite intervals: $f$ must have convergent power series on $[a,b]$
\item For semi-infinite integrals: $f$ must be entire and Laplace transformable
\item For integrals like $\int_0^{\infty} f(x) dx$: $f$ must have exponential-polynomial form $\sum c_j e^{-b_j x} x^{n_j}$ with $b_j > 0$
\item The domain of convergence may be restricted, and limits like $y \to 0$ may not exist without regularization
\end{itemize}

A quick inspection shows that these rules can be extended to integrals of the form $\int_{I} e^{-xy}f(y)g(y)dy$ as follows:

\begin{prop}
\label{prop1}
    We have the following equivalent evaluations 
    \begin{align*}
    \int_{I} e^{-xy}f(y)g(y)dy = f(-\partial_x)\int_{I} e^{-xy}g(y)dy\\
    = g(-\partial_x)\int_{I} e^{-xy}f(y)dy
    = f(-\partial_x)g(-\partial_x)\int_{I} e^{-xy}dy
    \end{align*}
\end{prop}
This result allows us to choose which function among $f,g$ or $fg$ should be extracted from the integral to form the pseudo-differential operator.

These powerful identities transform the potentially challenging task of integration into the application of a differential operator on a simple function, offering an alternative to traditional methods. As expected, and as several examples will show, there is a trade-off between the complexity of the pseudo-differential operator extracted from the integral and the complexity of the remaining integral.

In the following sections, we explore applications (section 2), verify compatibility with classical rules (section 3), and extend the method of integration by differentiation to broader settings (sections 4 to 7).

\section{Applications}
Now that  the basic framework is established, let us illustrate the method through a series of examples. These applications demonstrate how the operator approach can sometimes simplify otherwise challenging computations, and set the stage for a later extension to multivariate cases.

\subsection{A Classic Trigonometric Integral}
The evaluation
\[
\int_0^{\infty} \frac{\sin x}{x} \, dx = \frac{\pi}{2}
\]
may be obtained in many different ways, such as using Cauchy's residue theorem or classical Fourier analysis. The reader is referred to the entertaining article by Hardy~\cite{Hardy}, in which a "system of marking" is used to evaluate the complexity of various proofs. 

Our interest in the method of integration by differentiation was triggered by its efficiency in evaluating such integrals. The proof given by Kempf et al. \cite{Kempf} is as follows: choose $f(x) = \frac{\sin x}{x}$
in \eqref{Laplace limit} to produce
\[
\int_0^\infty \frac{\sin(x)}{x} \, dx 
= \lim_{y \to 0} \frac{1}{2i}  \frac{e^{-i\partial_y} - e^{i\partial_y} }{-\partial_y} \frac{1}{y}.
\]
The pseudo-differential operator $\frac{1}{\partial_y}$
being the integration operator, we deduce \( \frac{1}{\partial_y}\frac{1}{y}=\ln(y),\) where we use the principal determination of the logarithm, and the desired integral reads
\[
\lim_{y \to 0} \frac{-1}{2i} \left( e^{-i\partial_y} - e^{i\partial_y} \right) \ln(y)
= \lim_{y \to 0} \frac{ \ln(y + i) - \ln(y - i) }{2i} .
\]
This limit is easily evaluated as
\[
\frac{ \ln(i) - \ln(-i)}{2i}  = \frac{\pi}{2}.
\]

\subsection*{An Alternative Approach}
We propose here an evaluation that is slightly different from, but equivalent to, the one initially given by Kempf by choosing the easier pseudo-differential operator $f(-\partial_y)$ with $f(x)=\frac{1}{x}$:
\[
\int_{0}^{\infty} \frac{\sin x}{x} \, dx = \lim_{y \to 0} \left( \frac{1}{-\partial_y} \right) \int_{0}^{\infty} \sin x\, e^{-xy} \, dx.
\]
Next, denoting $\text{Im}$ as the imaginary part, this integral is identified as:
\[
\text{Im} \lim_{y \to 0} \left( \frac{1}{-\partial_y} \right) \int_{0}^{\infty} e^{-x(y-\imath)} \, dx
= \text{Im} \lim_{y \to 0} \left( \frac{1}{-\partial_y} \right) \left( \frac{1}{y - \imath} \right).
\]
The value of the integral is deduced as
\[
-\text{Im} \lim_{y \to 0} \log(y - \imath)
= -\text{Im} \log(-\imath)
= \frac{\pi}{2}.
\]

As noticed by Hardy in "Mr. Berry's first proof"~\cite{Hardy}, the previous steps involve several inversions of limits that should be justified with care.

\subsection{An integral by Ramanujan}
Entry 16(ii) p.264 of Ramanujan's second notebook \cite{RII} reads: 
\begin{prop}
For $n,p$ integers, the integral
\begin{equation*}
I_{n,p}=\int_{0}^{\infty}\frac{\sin^{2n+1}x}{x}\cos\left(2px\right)dx
\end{equation*}
is equal to 
\begin{equation}
\label{Inp}
I_{n,p}=\left(-1\right)^{p}\frac{\sqrt{\pi}}{2}\frac{\Gamma\left(n+1\right)\Gamma\left(n+\frac{1}{2}\right)}{\Gamma\left(n-p+1\right)\Gamma\left(n+p+1\right)}.
\end{equation}
\end{prop}
% It is also stated that 
% \[
% I_{n,p}=\int_{0}^{\infty}\frac{\sin^{2n+2}x}{x^{2}}\cos\left(2px\right)dx.
% \]
The proof produced by B. Berndt  in \cite{RII} is based on the fact that the sequence of integrals $I_{n,p}$ satisfies the recurrence  
% on the sequence of integrals $I_{n,p}$
\[
I_{n,p}=\frac{1}{2}I_{n-1,p}-\frac{1}{4}I_{n-1,p+1}-\frac{1}{4}I_{n-1,p-1}.
\]
% A surprising fact about this result is that the integral $I_{n,p}$ can in fact be computed by elementary means for all real values of the paramater $p$, 
We revisit this result using the method of integration by differentiation, producing the following generalization to real values of the parameter $p.$
\begin{prop}
For all integers $n$ 
and real numbers
%all real numbers
$p,$ the integral $I_{n,p}$ is equal to
\[
I_{n,p}=\frac{\left(-1\right)^{n}\pi}{2^{2n+1}}\sum_{k=0}^{2n+1}\binom{2n+1}{k}\left(-1\right)^{k}H\left(p+n+\frac{1}{2}-k\right)
\] 
where $H(x)$ is the Heaviside function
\[
H(x)=\begin{cases}
    1, \,\,x\ge 0\\
    0, \,\,x< 0.
\end{cases}
\]
Moreover, this evaluation coincides with Ramanujan's formula \eqref{Inp}
at integer values of $p.$
\end{prop}
% We propose here two different proofs for this result. 
% The first proof uses classic tools from Fourier analysis. The second proof will put the method of integration by differentiation to work.
%%%
\begin{proof}
% In the case of the Fourier integral, the method of integration by differentiation produces the formula
% \[
% \label{ibd}
% \int_{-\infty}^{\infty}f\left(y\right)e^{\imath xy}dy=2\pi f\left(-\imath\partial_{x}\right)\delta\left(x\right)
% \]
% with $\delta(x)$ the Dirac delta distribution. 
First, use symmetry to rewrite $I_{n,p}$ as
\[
I_{n,p}=\frac{1}{2}\int_{-\infty}^{\infty}\frac{\sin^{2n+1}x}{x}\cos\left(2px\right)dx
=\frac{1}{2}\int_{-\infty}^{\infty}\frac{\sin^{2n+1}x}{x}e^{\imath 2px}dx
\]
and perform the substitution
\( x \mapsto \frac{x}{2} \) to obtain:
\[
I_{n,p} = \frac{1}{2} \int_{-\infty}^{\infty} \frac{\sin^{2n+1}\left(\frac{x}{2}\right)}{x} e^{i p x} \, dx.
\]
Letting \( f(x) = \frac{\sin^{2n+1}\left(\frac{x}{2}\right)}{2x} \), identity \eqref{Fourier limit} produces
% \[
% I_{n,p} = 2\pi f\left(-i\partial_p\right)\delta(p)
% \]
\[
I_{n,p} 
% = 2\pi f\left(-i\partial_p\right)\delta(p)
= \pi \frac{\sin^{2n+1}\left(-i\frac{\partial_p}{2}\right)}{-i\partial_p} \delta(p)
= \frac{\pi}{i} \sin^{2n+1}\left(i\frac{\partial_p}{2}\right) H(p),
\]
where we have used the fact that the Heaviside function is an antiderivative for the Dirac delta distribution.
% : \(\frac{1}{\partial_p}\delta(p) = H(p).\)  

Now the power of the sine function is computed applying the binomial formula to obtain
\[
I_{n,p}=\imath\pi\left(\frac{e^{-\frac{\partial_{p}}{2}}-e^{\frac{\partial_{p}}{2}}}{2\imath}\right)^{2n+1}H(p)
\]
\[
=\left(-1\right)^{n+1}\frac{\pi}{2^{2n+1}}\sum_{k=0}^{2n+1}\binom{2n+1}{k}\left(-1\right)^{2n+1-k}e^{\left(-\frac{k}{2}+\frac{2n+1-k}{2}\right)\partial_{p}}H\left(p\right).
\]
Each pseudo-differential exponential $e^{a\partial_p}$ acts as a shift operator on the Heaviside function:
\[
e^{a\partial_p}H(p)=H(p+a)
\]
so that
\[
I_{n,p}=\frac{\left(-1\right)^{n}\pi}{2^{2n+1}}\sum_{k=0}^{2n+1}\binom{2n+1}{k}\left(-1\right)^{k}H\left(p+n+\frac{1}{2}-k\right),
\]
which completes the proof.
\end{proof}
% For example, take $n=1:$ 
% \[
% I_{1,p}=\int_{0}^{\infty}\frac{\sin^{3}x}{x}\cos\left(2px\right)dx=\begin{cases}
% \frac{\pi}{4}, & p=0\\
% -\frac{\pi}{8}, & p=\pm1\\
% 0, & p\ne0,\pm1
% \end{cases}
% \]
%  This result coincides with Ramanujan's formula:
% \[
% I_{1,p}=\left(-1\right)^{p}\frac{\sqrt{\pi}}{2}\frac{\Gamma\left(2\right)\Gamma\left(\frac{3}{2}\right)}{\Gamma\left(2-p\right)\Gamma\left(2+p\right)}
% \]
The reader is invited to check that, in the case where $p$ is an integer, this formula for $I_{n,p}$ coincides with the value found by Ramanujan.

\section{Compatibility with Classical Integration Rules}

We have shown the power of this symbolic method through two simple examples in the one-dimensional case. Before we extend it to broader settings, we first verify that the method is compatible with the fundamental rules of calculus. In particular, let us check its compatibility with integration by parts and with change of variables.

\subsection{Integration by Parts}
We consider the symbolic integration rule
\[
\lim_{y \to 0} f(-\partial_y) \int_{0}^{\infty}g(x)e^{-xy}dx
= \int_{0}^{\infty}f(x)g(x)dx
\]
and wonder if it is compatible with the rule of integration by parts
\[
\int f(x)g'(x)dx
= f(x)g(x) - \int f'(x)g(x)dx.
\]
The method of integration by differentiation evaluates the left-hand side of this rule as

\[
\int f\left(x\right)g'\left(x\right)dx
=\lim_{y\to0}f\left(-\partial_{y}\right)\int g'\left(x\right)e^{-xy}dx
\]
a result that, after integration by parts,  coincides with
\[
\lim_{y\to0}f\left(-\partial_{y}\right)
\left( \left[g\left(x\right)e^{-xy}\right]
+
\int yg\left(x\right)e^{-xy}dx
\right)\]\[
=\lim_{y\to0}f\left(-\partial_{y}\right)\left[g\left(x\right)e^{-xy}\right]
+\lim_{y\to0}f\left(-\partial_{y}\right)\int yg\left(x\right)e^{-xy}dx
\]
The first term evaluates as
\[
\lim_{y\to0}f\left(-\partial_{y}\right)\left[g\left(x\right)e^{-xy}\right]=\lim_{y\to0}f\left(x\right)g\left(x\right)e^{-xy}=f\left(x\right)g\left(x\right)
\]
while the second term is computed noticing   that
\[
f\left(-\partial_{y}\right) y e^{-xy}
=-f\left(-\partial_{y}\right)\partial_x e^{-xy}
= -\partial_x f\left(-\partial_{y}\right)e^{-xy} \]
which is identified as
\[
 -\partial_x \left\{ f(x)e^{-xy}\right\}
 =-e^{-xy}\left(f'(x)-yf(x)\right)
\]
so that this second term is
\[
\lim_{y \to 0}\int - g(x) e^{-xy} (f'(x)-yf(x)) dx
= -\int g(x) f'(x) dx
\]
showing that the integration by parts rule is satisfied.
% Back to the integration by parts proof, we deduce
% \[
% \lim_{y\to0}f\left(\partial_{y}\right)\int xg\left(x\right)e^{xy}dx=\int g\left(x\right)f'\left(x\right)dx
% \]
% and
% \[
% \lim_{y\to0}f\left(\partial_{y}\right)\left\{ \left[g\left(x\right)e^{xy}\right]-\int yg\left(x\right)e^{xy}dx\right\} =\int g\left(x\right)f'\left(x\right)dx
% \]
% as expected.

\subsection{Change of Variables}
We want to check that the change of variables formula
\[
\int f\left(\phi\left(x\right)\right)e^{-\phi\left(x\right)y}\phi'\left(x\right)dx=\int f\left(u\right)e^{-uy}du
\]
holds when both sides are computed using the method of integration
by differentiation.

The  integral on the right is evaluated by this method as 
\[
\int f\left(u\right)e^{-uy}du=f\left(-\partial_{y}\right)\int e^{-uy}du.
\]
while the  integral on the left is evaluated by this method as
\[
\int f\left(\phi(x)\right)e^{-y\phi(x)}\phi'(x)dx
=f\left(-\partial_{y}\right) \int e^{-y\phi(x)}\phi'(x)dx
=f\left(-\partial_{y}\right) \int e^{-yu}du.
\]
This shows that the method of integration by differentiation is compatible with the change of variables rule.

\section{Extension to the Multivariate Case}

Having verified its compatibility with classical integration rules, we now extend the method of integration by differentiation to the multivariate setting.

\subsection{Tensorization}

One nice feature of the method is that it naturally tensorizes across multiple variables. In the two-dimensional case, for instance, we propose the following formula
\begin{equation}
\label{bidimensional}
\int_0^\infty \int_0^\infty f(x, y) e^{-u x - v y} \, dx \, dy = f(-\partial_u, -\partial_v) \left( \frac{1}{uv} \right),
\end{equation}

for \(f\) an analytic function in two variables.
More generally, the method extends to the case of \(n\) variables as follows:
\[
\int_{[0,\infty)^n} f(x_1, \dots, x_n) e^{-\sum_{i=1}^n u_i x_i} \, dx_1 \cdots dx_n = f(-\partial_{u_1}, \dots, -\partial_{u_n}) \prod_{i=1}^n \frac{1}{u_i}.
\]

This natural extension to multiple variables presents a significant advantage of the method. 
% While many integration techniques become exponentially more complex, the integration by differentiation approach maintains essentially the same structure.
Higher-dimensional integrals, such as $n-$th variate Laplace integrals, are usually difficult to compute, and are hard to find in the literature: we may recommend  Chapters 3 of the first and second volumes of Prudnikov's tables \cite{Pr1} and the less-known reference  \cite{poli}  for the two-dimensional case.

%\subsection{General Multivariate Formulation}

As previously, for an arbitrary subset \(I\in\mathbb{R}^n\) over which the integral converges, integration by differentiation is expressed by
\[
\int_I f(x_1, \dots, x_n) e^{-\sum_{i=1}^n u_i x_i} \, dx_1 \cdots dx_n = f(-\partial_{u_1}, \dots, -\partial_{u_n}) \int_I e^{-\sum_{i=1}^n u_i x_i} \, dx_1 \cdots dx_n.
\]

and its variant
\[
\int_I f(x_1, \dots, x_n)  \, dx_1 \cdots dx_n = \lim_{\mathbf{u}\to \mathbf{0}} f(-\partial_{u_1}, \dots, -\partial_{u_n}) \int_I e^{-\sum_{i=1}^n u_i x_i} \, dx_1 \cdots dx_n.
\]

\subsection{The case of rotational invariance}
A dedicated formula can be provided in the case where the integrand is rotationally invariant, i.e. is a function of the Euclidean norm $\vert \mathbf{x} \vert$ of $\mathbf{x}$ only \footnote{ Boldface letters represent $n-$dimensional vectors and $\mathbf{b}\mathbf{x}$ denotes the usual inner product $\sum_{i=1}^n b_ix_i.$}: an integral of the form 
\[
\int_{\mathbf{R}^{n}}f\left(\vert\mathbf{x}\vert\right)e^{-u\vert\mathbf{x}\vert}d\mathbf{x},\,\,u>0,
\]
can be computed as
\[
\int_{\mathbf{R}^{n}}f\left(\vert\mathbf{x}\vert\right)e^{-u\vert\mathbf{x}\vert}d\mathbf{x}=f\left(-\partial_{u}\right)\int_{\mathbf{R}^{n}}e^{-u\vert\mathbf{x}\vert}d\mathbf{x}
\]
The integral on the right is transformed using the multivariate change of variable
$\mathbf{y}=u\mathbf{x}$ with Jacobian $\frac{1}{u^n}$ as
\[
\int_{\mathbf{R}^{n}}e^{-u\vert\mathbf{x}\vert}d\mathbf{x}=\frac{1}{u^{n}}\int_{\mathbf{R}^{n}}e^{-\vert\mathbf{y}\vert}d\mathbf{y}
\]
and the remaining integral is computed using entry (3.3.2.1) in \cite{Pr1} 
% \[
% \int_{\left[0,\infty\right)^{n}}f\left(\vert\mathbf{x}\vert^{2}\right)d\mathbf{x}=\frac{2\pi^{\frac{n}{2}}}{\Gamma\left(\frac{n}{2}\right)}\int_{0}^{\infty}t^{n-1}f\left(t^{2}\right)dt
% \]
as
\[
\int_{\mathbf{R}^{n}}e^{-\vert\mathbf{y}\vert}d\mathbf{y}
% =\frac{2\pi^{\frac{n}{2}}}{\Gamma\left(\frac{n}{2}\right)}\int_{0}^{\infty}t^{n-1}e^{-t}dt
=\frac{2\pi^{\frac{n}{2}}\Gamma\left(n\right)}{\Gamma\left(\frac{n}{2}\right)}.
\]
We deduce the version of integration by differentiation in the rotationally invariant case as
\[
\int_{\mathbf{R}^{n}}f\left(\vert\mathbf{x}\vert\right)e^{-u\vert\mathbf{x}\vert}d\mathbf{x}=\frac{2\pi^{\frac{n}{2}}\Gamma\left(n\right)}{\Gamma\left(\frac{n}{2}\right)}f\left(-\partial_{u}\right)\frac{1}{u^{n}}
\]
and its limit version
\[
\int_{\mathbf{R}^{n}}f\left(\vert\mathbf{x}\vert\right)d\mathbf{x}=\frac{2\pi^{\frac{n}{2}}\Gamma\left(n\right)}{\Gamma\left(\frac{n}{2}\right)}\lim_{u\to 0}f\left(-\partial_{u}\right)\frac{1}{u^{n}}.
\]
A simple example is given by the integral
\[
I_{n}=\int_{\mathbf{R}^{n}} \frac{\sin \vert \mathbf{x}\vert  }{\vert \mathbf{x}\vert^n} d\mathbf{x} ,\,\,n\ge 1,
\]
The action of the operator $f(-\partial_u)=\frac{\sin(-\partial_u)}{(-\partial_u)^n}$ on the function $\frac{1}{u^n}$  is computed as previously in the case $n=1$ and produces
 \[
\frac{1}{(n-1)!}\frac{\log(u+\imath)-\log(u-\imath)}{2\imath}
\]
so that, after simplification,
\[
I_{n}=\frac{\pi^{\frac{n}{2}+1}}{\Gamma\left(\frac{n}{2}\right)},\,\,n\ge 1.
 \]

This result agrees  with the one obtained by direct application of the limit case $r\to +\infty$  in  formula (3.3.2.1) in \cite{Pr1}:
\[
\int_{\vert\mathbf{x}\vert^2\le r^2} f(\vert \mathbf{x}\vert ^2)d\mathbf{x} = \frac{2 \pi^{n/2} }{\Gamma(\frac{n}{2})} \int_0^{r}t^{n-1}f(t^2)dt.   
\]

In the next section, we will illustrate the efficiency of these rules through evaluation of multivariate Laplace integrals and computations of volumes of simplexes.

\section{Multivariate Applications}
\subsection{a family of bivariate Laplace transforms}
The following evaluation appears as entry (3.1.3.8) in \cite{Pr1}: for $u>0,v>0$ such that $u \ne v,$ \footnote{this formula corrects a sign mistake in \cite[3.1.3.8]{Pr1}}
\[
\int_{0}^{\infty}\int_{0}^{\infty}\frac{1}{x+y}e^{-ux-vy}dxdy=\frac{\log u-\log v}{u-v},
\]
and can be obtained as follows:
first apply the bidimensional integration by differentiation rule \eqref{bidimensional} to obtain
\[
\int_{0}^{\infty}\int_{0}^{\infty}\frac{1}{x+y}e^{-ux-vy}dxdy = 
\frac{1}{-\partial_{u}-\partial_{v}}\frac{1}{uv}.
\]
Now we use the Euler integral to express the pseudo-differential operator as
\[
\frac{1}{\partial_{u}+\partial_{v}} = \int_{0}^{\infty}e^{-w\left(\partial_{u}+\partial_{v}\right)}dw
\]
so that
\[
\frac{1}{-\partial_{u}-\partial_{v}}\frac{1}{uv} 
=-\int_{0}^{\infty}e^{-w\left(\partial_{u}+\partial_{v}\right)}dw\frac{1}{uv}=
-\int_{0}^{\infty} \frac{dw}{(u-w)(v-w)}dw.
\]
Partial fraction decomposition produces the desired result
\[
% =\frac{1}{v-u}\left[\log \frac{u+w}{v+w}\right]
\frac{\log (-v) - \log (-u)}{v-u}=\frac{\log (v) - \log (u)}{v-u}.
\]
The method extends naturally to the family of  integrals \cite[3.1.3.7]{Pr1}
\[
\int_0^{\infty} \int_0^{\infty}
(x+y)^{\nu - 1}
e^{-ux-vy}dxdy = \Gamma(\nu) \frac{u^\nu - v^\nu}{(u-v)(uv)^\nu}
\]
with $\text{Re}(\nu)>0.$ The details are left to the reader.
\subsection{A Multivariate Euler-Like Integral}
The following multivariate integral appears as Entry 3.3.5.4 in \cite{Pr1}: with $a_k,b_k,\nu_k$ and $\mu$ positive real numbers,
 \[\begin{aligned}\int _{0}^{\infty } \dots \int _{0}^{\infty }\dfrac{e^{-\mathbf{b}\mathbf{x}}}{\left( a_{0}+\mathbf{a}\mathbf{x}\right) ^{\mu }}\left(\prod ^{n}_{k=1}x_{k}^{\nu _{k}-1}\right)dx_{1} \ldots dx_{n}\\
=\Gamma \begin{pmatrix}
\nu_{1},  & \ldots  &, \nu _{n} \\
 & \mu  & 
\end{pmatrix}\int _{0}^{\infty }e^{-a_{0}t}t^{\mu -1}\prod ^{n}_{k=1}\left( b_{k}+a_{k}t\right) ^{-\nu_{k}}dt.\end{aligned}\]
This identity shows a remarkable  symmetry: the reduction of the multivariate integral on the left to a univariate integral on the right results in the exchange of parameters $\mu$ to $\nu$.
The method of integration by differentiation suggests expressing the multivariate integral as
\[
% \int _{0}^{\infty } \dots \int _{0}^{\infty }\dfrac{e^{-\mathbf{b}\mathbf{x}}}{\left( a_{0}+\mathbf{a}\mathbf{x}\right) ^{\mu }}\left(\prod ^{n}_{k=1}x_{k}^{\nu _{k}-1}\right)dx_{1} \ldots dx_{n}
% =
\dfrac{1}{\left( a_{0}-\mathbf{a}\mathbf{\partial_b}\right) ^{\mu }}
\int _{0}^{\infty } \dots \int _{0}^{\infty }
e^{-\mathbf{b}\mathbf{x}}\left(\prod ^{n}_{k=1}x_{k}^{\nu _{k}-1}\right)dx_{1} \ldots dx_{n}
\]
with the benefit that the resulting multivariate integral now factors as the product of $n$ Euler integrals
\[
\prod_{k=1}^{n}
\int_{0}^{\infty}
e^{-b_k x_k}x_k^{\nu_k-1}dx_{k} = \prod_{k=1}^{n}
\frac{\Gamma(\nu_k)}{b_k^{\nu_k}}.
\]
It remains to compute the action of the pseudo-differential operator $\dfrac{1}{\left( a_{0}-\mathbf{a}\mathbf{\partial_b}\right) ^{\mu }}$ on this function: we'll use again an Euler integral representation, namely
\[
\dfrac{1}{\left( a_{0}-\mathbf{a}\mathbf{\partial_b}\right) ^{\mu }} 
= \frac{1}{\Gamma(\mu)}\int_{0}^{\infty}
e^{-t(a_0-\mathbf{a}\mathbf{\partial_b})}t^{\mu -1}dt
\]
to obtain the desired integral as
\[
\int_{0}^{\infty}
e^{-t(a_0-\mathbf{a}\mathbf{\partial_b})}t^{\mu -1}dt
\prod_{k=1}^{n}
\frac{\Gamma(\nu_k)}{b_k^{\nu_k}}.
\]
Since
\[
e^{t\mathbf{a}\mathbf{\partial_b}}=\prod_{k=1}^{n}e^{t a_k \partial_{b_k}},
\]
we deduce the final integral as
\[
\frac{\prod_{k=1}^{n}
\Gamma(\nu_k)}{\Gamma(\mu)}
\int_{0}^{\infty}
e^{-ta_0}t^{\mu -1}
\prod_{k=1}^{n}
(a_{k}t+b_k)^{-\nu_k}dt.
\]

\section{Computation of Volumes}

We now explore a particularly unexpected application of the integration by differentiation rule that demonstrates the method's scope beyond the field of analytic functions.  The following example illustrates how the volume of a standard simplex can be computed through the Laplace transforms of their indicator functions, using the multivariate version of the integration by differentiation method.

\subsection{Laplace Transform of the Simplex Indicator}

In this section, we are concerned with the Laplace transform of the indicator function $\mathrm{1}_{S_n}$ of the $n-$dimensional simplex
\[
S_n = \left\{ \mathbf{x} \in \mathbb{R}^n : x_i \ge 0 \text{ for all } i, \; \sum_{i=1}^n x_i \le 1 \right\}.
\]
The indicator function is the function $\mathrm{1}_{S_n}:\mathbb{R}^n \to \{0,1\}$ defined by
\[
\mathrm{1}_{S_n}(\mathbf{x}) = \begin{cases}
1, & \mathbf{x} \in S_n, \\
0, & \text{else},
\end{cases}
\]
so that its Laplace transform
\[
I_{S_n}(\mathbf{a}) = \int_{\mathbb{R}^n} \mathrm{1}_{S_n}(\mathbf{x}) e^{-\mathbf{a} \cdot \mathbf{x}} d\mathbf{x}
\]
coincides with the Laplace transform computed on the simplex
\[
I_{S_n}(\mathbf{a}) = \int_{S_n} 
e^{-\mathbf{a} \cdot \mathbf{x}} d\mathbf{x}.
\]
Notice that its value at $\mathbf{a} = 0$ is the volume of the simplex $S_n$.

Our goal is to show that integration by differentiation allows to compute very easily this classic but difficult integral that appears as entry 3.3.4.17 in \cite{Pr1} :
\begin{equation}
\label{ISn(a)}
    I_{S_n}(\mathbf{a}) = \sum_{k=1}^n \frac{e^{-a_k} - 1}{(-a_k) \varphi'(-a_k)},
\end{equation}

where
\begin{equation}
\label{varphi}
\varphi(x) = \prod_{i=1}^{n}(x+a_i).
\end{equation}
Using the method of integration by differentiation, this integral is expressed as
\begin{equation}
\label{indicator}    
I_{S_n}(\mathbf{a}) = \mathrm{1}_{S_n}(-\partial_{a_1}, \dots, -\partial_{a_n}) \left( \frac{1}{a_1 \dots a_n} \right).
\end{equation}

Looking first at the obvious $1-$dimensional case will allow us to introduce the necessary tools to proceed easily to the arbitrary $n-$dimensional case.
\subsection{One-dimensional case}

In dimension $1$, the simplex $S_1$ coincides with the unit interval $[0,1]$  and the desired Laplace transform is simply
\[
I_{[0,1]}(a) = \int_0^1 e^{-ax} dx = \frac{1 - e^{-a}}{a}.
\]
Let us see how it can be computed using formula \eqref{indicator}.
The indicator function is
\[
\mathrm{1}_{S_1}(x) = H(1-x),
\]
with $H$ the Heaviside function,
% \[
% H(x) = \begin{cases}
% 1, & x \ge 0, \\
% 0, & x < 0
% \end{cases}
% \]
so that \eqref{indicator} reads
\[
I_{[0,1]}(a) = H(1 + \partial_a) \left( \frac{1}{a} \right).
\]
Now we need a representation of the Heaviside function that will allow us to easily compute the action of the operator $H(1+\partial_a)$ on $\frac{1}{a}.$
This representation is provided by the well-known Fourier integral decomposition of the sign function
\[
\text{sign}(x) = \frac{1}{\pi} \int_{\mathbb{R}} \frac{\sin(xy)}{y} dy,
\]
and the identity $H(x) = \frac{1}{2}(1 + \text{sign}(x))$, from which we deduce
\begin{equation}
\label{H}
H(x) = \frac{1}{2}\left(1 + \frac{1}{\pi} \int_{\mathbb{R}} \frac{\sin(xy)}{y} dy\right).
\end{equation}
Thus, the indicator function is computed as
\[
I_{[0,1]}(a) = H(1 + \partial_a) \left( \frac{1}{a} \right) = \frac{1}{2a} + \frac{1}{2\pi} \int_{\mathbb{R}} \frac{\sin\left((1+\partial_a)y\right)}{y} dy \left( \frac{1}{a} \right).
\]
Writing the sine term as the imaginary part
$
% \sin((1+\partial_a)y) 
% = \frac{1}{2i}\left(e^{i(1+\partial_a)y} - e^{-i(1+\partial_a)y}\right) 
\text{Im} \left(e^{iy}e^{iy\partial_a}\right),
% - e^{-iy}e^{-iy\partial_a},
$
the action  of this operator  on  $1/a$ is deduced as
\[
\sin((1+\partial_a)y) \frac{1}{a} 
% = \frac{1}{2i}\left(\frac{e^{iy}}{a+iy} - \frac{e^{-iy}}{a-iy}\right)
= \text{Im}\left( \frac{e^{iy}}{a+iy} \right) = \frac{a \sin y - y \cos y}{a^2 + y^2},
\]
so that
\[
\int_{\mathbb{R}} \frac{\sin((1+\partial_a)y)}{y} dy \frac{1}{a} = \int_{\mathbb{R}} \frac{a \sin y - y \cos y}{y(a^2 + y^2)} dy, 
\]
an integral easily computed as $\pi \frac{1-2e^{-a}}{a}$ to produce
\[
I_{[0,1]}(a) = \frac{1}{2a} + \frac{1}{2\pi}\left(\pi \frac{1-2e^{-a}}{a}\right) = \frac{1-e^{-a}}{a},
\]
a result that agrees with the direct computation.

\subsection{Arbitrary dimensional case}

The $n-$th dimensional simplex $S_n$ has indicator function
\[
\mathrm{1}_{S_n}(\mathbf{x}) = H(1-\sum_{i=1}^n x_i).
\]
Thus,
\[
I_{S_n}(\mathbf{a}) = H(1+ \sum_{i=1}^n \partial_{a_i}) \left( \frac{1}{a_1 \dots a_n} \right).
\]
Using again the integral representation \eqref{H}  and expanding as previously produces
\[
\sin((1+\sum_{i=1}^n \partial_{a_i})y) \left( \frac{1}{a_1 \dots a_n} \right) = \text{Im} \left( \frac{e^{iy}}{(a_1+iy)\dots(a_n+iy)} \right)
= \text{Im} \left( e^{iy}\varphi(\imath y)\right).
\]
with $\varphi(x)$ defined as in \eqref{varphi}. Partial fractions decomposition  of the function $\varphi(\imath y)$
and integration produce the desired result  \eqref{ISn(a)}.
% \[
% I_{S_2}(a_1,a_2) = \frac{1}{a_1a_2} - \frac{1}{a_2-a_1} \left( \frac{e^{-a_1}}{a_1} - \frac{e^{-a_2}}{a_2} \right).
% \]

% This matches the form given by Prudnikov after algebraic simplifications.
% Finally, letting $a_1, a_2 \to 0$, we recover the volume of the simplex $\text{Vol}(S_2) = \frac{1}{2}$.

\subsection{Another application of the indicator function}
The next example provides further insight about the way the indicator function can be used in the method of integration by differentiation. 
For positive parameters $\alpha,\beta,\gamma$, let us consider the integral over the three-dimensional simplex
\[
I_{\alpha,\beta,\gamma}=\int_{S_3}
x^{\alpha-1}
y^{\beta-1}
z^{\gamma-1}
f(x+y+z)dxdydz\]
where $f$ is an arbitrary function such that the integral is defined. Our goal is to express this integral as a one-dimensional integral by exploiting the invariance of the function $f(x+y+z)$  on hyperplanes. Let us consider the parameterized integral
\[
J_{\alpha,\beta,\gamma}(u)=\int_{S_3}x^{\alpha-1}
y^{\beta-1}
z^{\gamma-1}
f(x+y+z)e^{-u(x+y+z)}
dxdydz,\,\,u>0,
\]
and notice that, after insertion of the proper indicator function, it can be expressed as an integral over the whole three-dimensional space
\[
J_{\alpha,\beta,\gamma}(u)=\int_{\mathbb{R}^3}H(1-x-y-z)x^{\alpha-1}
y^{\beta-1}
z^{\gamma-1}
f(x+y+z)e^{-u(x+y+z)}
dxdydz.
\]
The  integration by differentiation technique produces the  formula
\[
J_{\alpha,\beta,\gamma}(u)=H(1+\partial_u)
f(-\partial_u)
\int_{\mathbb{R}^3}x^{\alpha-1}
y^{\beta-1}
z^{\gamma-1}
e^{-u(x+y+z)}
dxdydz
\]
where the resulting integral decouples as the product of three Euler integrals:
\[
\int_{\mathbb{R}^3}x^{\alpha-1}
y^{\beta-1}
z^{\gamma-1}
e^{-u(x+y+z)}
dxdydz =
\frac{\Gamma(\alpha)\Gamma(\beta)\Gamma(\gamma)}{u^{\alpha+\beta+\gamma}}.
\] 
This quantity can in turn be expressed as the Euler integral:
\[
\frac{\Gamma(\alpha)\Gamma(\beta)\Gamma(\gamma)}{\Gamma(\alpha+\beta+\gamma)}
\int_{0}^{+\infty}e^{-ut}t^{\alpha+\beta+\gamma-1}dt.
\]
The integral $J_{\alpha,\beta,\gamma}(u)$ is finally computed as
\[
J_{\alpha,\beta,\gamma}(u)=
\frac{\Gamma(\alpha)\Gamma(\beta)\Gamma(\gamma)}{\Gamma(\alpha+\beta+\gamma)}
H(1+\partial_u)
f(-\partial_u)
\int_{0}^{+\infty}e^{-ut}t^{\alpha+\beta+\gamma-1}dt.
\]
The action of the pseudo-differential operator $H(1+\partial_u)
f(-\partial_u)$
on the integral restricts its domain to the interval $\left[0,1\right]$
and reinserts the function $f$ in the integrand: this produces 
\[
J_{\alpha,\beta,\gamma}(u)=
\frac{\Gamma(\alpha)\Gamma(\beta)\Gamma(\gamma)}{\Gamma(\alpha+\beta+\gamma)}
\int_{0}^{1}f(t)e^{-ut}t^{\alpha+\beta+\gamma-1}dt
\]
and the desired result as the limit as $u\to 0$
\[
I_{\alpha,\beta,\gamma}=
\frac{\Gamma(\alpha)\Gamma(\beta)\Gamma(\gamma)}{\Gamma(\alpha+\beta+\gamma)}
\int_{0}^{1}f(t)t^{\alpha+\beta+\gamma-1}dt.
\]
This identity appears as Entry (3.2.2.3) in \cite{Pr1}.

Let us notice that, following the same steps, this identity extends to the $n-$variate case: 
\[
\int_{S_n}
f(\sum_{i=1}^n x_i)
\prod_{i=1}^n
\frac{x_i^{a_i}}{\Gamma(a_i)}
\frac{dx_i}{x_i}
= 
\int_{0}^{1}f(t)
\frac{t^{\sum_{i=1}^n a_i}}{\Gamma(\sum_{i=1}^n a_i)}\frac{dt}{t}.
\]
Moreover, let  us highlight the two remarkable ways the pseudo-differential operator $H(1+\partial_u)
f(-\partial_u)$ acts in this computation: it  produces a natural decoupling of the three-dimensional integral and a transparent handling of the function $f$.
%\section{Advanced Extensions: Beyond First-Order Operators}
% \subsection{Second-Order Differential Operators}
% \subsection{Fractional Differential Operators}
\section{$q-$integration by $q-$differentiation}
\subsection{introduction to $q-$calculus}

The theory of $q-$calculus deals with non-commutative operators $A$ and $B$ that satisfy $AB=qBA$ for a real parameter $q$ usually chosen such that $\vert q \vert <1.$ The computation of the operator $(A+B)^n$,  with $n$ an integer, reveals the q-factorial number $[j]!_q=\prod_{k=1}^{j}[k]_q$, where the q-integer number $[j]_q$ is defined by $[j]_q=\frac{1-q^j}{1-q}.$ Notice that, as $q \to 1,$ these numbers  coincide with the usual factorial and integers numbers. The reader is referred to the wonderful book by V. Kac  and P. Cheung for more details about the $q-$calculus theory.

The theory introduces next the $q-$differential operator \cite[(1.5)]{Kac}
\[
D_q f(x) = \frac{f(qx)-f(x)}{(q-1)x},
\]
and the $q-$exponential function
\[e_q(x) = \sum_{j\ge0} \frac{x^j}{\left[j\right]!_q}
=\prod_{n\ge1}\frac{1}{(1-x(1-q)q^n)}
\]
that is easily shown to satisfy the differential equation \cite[(9.11)]{Kac}
\[
D_q e_q(x) =e_q(x),
\]
and more generally, for a real constant $a,$
\[
D_q e_q(ax)  =a e_q(ax).
\]
as a consequence of
\[
D_q e_q(ax) =\frac{e_q(qax)-e_q(ax)}{(q-1)x}
=\frac{e_q(qax)-e_q(ax)}{(q-1)ax}\frac{qax-ax}{(q-1)x}.
\]
\subsection{Jackson's integral}
Jackson's integral appears as a variant of the Riemann sum where the sampling points are chosen according to a geometric  distribution:
\[
\int f(x)d_{q}x
=(1-q)\sum_{n\ge0}xq^nf(xq^n)
\]
with the definite versions
\[
\int_{0}^{b} f(x)d_{q}x
=(1-q)b \sum_{n\ge0}q^nf(q^n b),
\]
\[
\int_{a}^{b} f(x)d_{q}x = 
\int_{0}^{b} f(x)d_{q}x 
-
\int_{0}^{a} f(x)d_{q}x. 
\]
and
\[
\int_{0}^{\infty} f(x)d_{q}x
=(1-q) \sum_{n\ge0}q^nf(q^n).
\]
 
It is an easy exercise \cite[Chapter 18]{Kac} to show that  Jackson's integral is the inverse operator  of the $q-$differential operator $D_q$:
\[
D_q \int f(x)d_{q}x = f(x).
\]

For  a function $f$  such that the Jackson integral $\int_{a}^{b} f(x) d_qx$ exists, we deduce the extension of the integration by differentiation rule to Jackson's integrals as
\begin{equation}
\int_{a}^{b} f(x) d_qx = \lim_{y \to 0} f(D_q)\frac{e_{q}(by)-e_{q}(ay)}{y}.
\label{q-integration}
\end{equation} 

% One possible application would be to derive Theorem 3.2 from \cite{Hu} starting with this identity.

% \[\int_{0}^{1} \Phi_{k}(z,s,a) d_qa = \frac{1}{[1-s]_q} + \sum_{r=0}^{k-1} 
% \binom{k}{r}\sum_{l=0}^{\infty} \binom{-s}{l} \frac{\text{Li}_{k-r}(z,s+l)}{[l+1]_q}, \quad |z| \leq 1 , \text{Re}(s)<1.\]

% Since \begin{align}\label{1}
%     \lim_{y\to 0}D_{q}^{j}  \frac{e_{q}(y) - 1}{y} = \frac{1}{[j+1]_q}
% \end{align}
% and we combine this with binomial theorem. Since $\Phi_{k}(z,s,a)$ has a singularity at $a=0$, we rewrite this as 
% $\Phi_{k}(z,s,a) = z \Phi_{k}(z,s,a+1) + \frac{1}{a^{-s}} $.
\subsection{Example: $q-$integration by $q-$differentiation of Hurwitz zeta function}
Let us illustrate the potential of this rule to recover the Jackson integral of the Hurwitz zeta function
\[
\zeta(s,a)=\sum_{n\ge 0} \frac{1}{(a+n)^s}
\]
as first derived by Kurokawa et al. \cite[Thm. 2]{kurokawa}: for $\text{Re}(s)>1,$
\[ \int_{0}^{1} \zeta(s,a) d_q a = \frac{1}{[1-s]_q} + \sum_{k=0}^{\infty} \binom{-s}{k} \frac{\zeta(s+k)}{[k+1]_q}.\]
\begin{proof}
We start with the observation that
\[
\zeta(s,a)=a^{-s}+\zeta(s,a+1)
\]
so that
\begin{align*}
    \int_{0}^{1} \zeta(s,a) d_qa = \int_{0}^{1} a^{-s} + \zeta(s,1+a) d_qa &= \frac{1}{[1-s]_q} +  \int_{0}^{1} \zeta(s,1+a) d_qa.
\end{align*}
The remaining integral is computed using formula \eqref{q-integration} with $f(x)=\zeta(s,1+x)$ so that  we need to expand the operator
\begin{align*}
    &\zeta(s,1+D_q) = \sum_{p=1}^{\infty} \frac{1}{p^s \left(1+ \frac{D_q}{p}\right)^s} = \sum_{p=1}^{\infty} \sum_{k=0}^{\infty} \binom{-s}{k} \frac{D_{q}^{k}}{p^{s+k}} \\
    &= \sum_{k=0}^{\infty} \binom{-s}{k}  \sum_{p=1}^{\infty} \frac{1}{p^{s+k}} D_{q}^{k}= \sum_{k=0}^{\infty} \binom{-s}{k}  \zeta(s+k)D_{q}^{k}.
\end{align*}
% Using equation \eqref{q-integration} with $a=0$ and $b=1$, we  deduce
We deduce the value of the integral
\begin{align*}
    \int_{0}^{1} \zeta(s,1+a) d_qa = \lim_{y \to 0} \sum_{k=0}^{\infty} \binom{-s}{k}  \zeta(s+k)
    D_{q}^{k}    \frac{e_{q}(y)-1}{y} = \sum_{k=0}^{\infty} \binom{-s}{k} \frac{\zeta(s+k)}{[k+1]_q}.
\end{align*}
where $
D_{q}^{k}    \frac{e_{q}(y)-1}{y}
$
is identified as the $k-$th Taylor coefficient (see \cite[Ch. 4]{Kac}) of the function $\frac{e_{q}^{y}-1}{y}$, hence as $\frac{1}{[k+1]_q}$
and we deduce the result.

\end{proof} 
% The equation
% \begin{equation}
% \int_{a}^{b} f(x) d_qx = \lim_{y \to 0} f(D_q)\frac{e_{q}(by)-e_{q}(ay)}{y}.
% \label{q-integration}
% \end{equation} 
% demonstrates that this method extends naturally to q-calculus.

\section{Conclusion}
Integration by differentiation is one of several symbolic methods for evaluating integrals, standing alongside techniques such as Ramanujan’s Master Theorem. 
As we have shown, this method transforms the often intractable task of evaluating an integral into the problem of applying a pseudo-differential operator to a simpler function. Although the method often offers computational advantages, it usually presents two main challenges. First, one must carefully and creatively find a representation of the operator that is suitable for the given problem. Second, as for most symbolic methods, the result must be confirmed analytically after computation.

Despite these challenges, we were positively impressed by the versatility of this method;  its ability to handle indicator functions and simplify the domain of integration is a particularly useful asset.

Looking forward, we see promising ways to extend this method. For example, \cite{Cohen}, 
introduces kernals $u(x,y)$ that are eigenfunctions of non standard pseudo-differential operators $D_{x}$, satisfying
\[
D_x u(x,y)=yu(x,y).
\]
Using this property, we can deduce the symbolic integration rule:
\[
\int f(y)g(y)dy = \lim_{x\to 0}f(D_x)\int u(x,y)g(y)dy. 
\]
Two examples of such kernels are:
\begin{itemize}
    \item \( u(x, y) = e^{\imath(xy - \frac{x^2}{2})} \), with associated operator \( D_x = x - \imath \frac{d}{dx} \);
    \item \( u(x, y) = e^{-\imath x y} e^{\imath \frac{x^3}{3}} \), with associated operator \( D_x = \imath \frac{d}{dx} + x^3 \).
\end{itemize}
These constructions suggest a broader symbolic framework for evaluating integrals. We plan to explore these ideas further in future work.

\section{Acknowledgments}
The authors thank Parth Chavan for providing the example related to the Jackson integral, and Michael Joyce for providing the opportunity to start this project.
% \bibliographystyle{elsarticle-num}
% \bibliography{references}

\begin{thebibliography}{25}


%\bibitem{AD} T.~Agoh and K.~Dilcher, Integrals of products of Bernoulli 
%polynomials, {\it J. Math. Anal. Appl.} {\bf 381} (2011), no.~1, 10--16.
%
%\bibitem{AS} M.~Abramowitz and I.~A.~Stegun, {\it Handbook
%of Mathematical Functions}, National Bureau of Standards, 1964.
%
%\bibitem{At} F.~V.~Atkinson, The mean-value of the Riemann zeta function,
%{\it Acta Math.} {\bf 81} (1949), 353--376. 
%e

\bibitem{RII}
B. C. Berndt, \textit{Ramanujan's Notebook, Part II}, Springer, 1989.

\bibitem{Cohen}
L. Cohen, The Weyl operator and its Generalization, Birkh\"{a}user, 2013.

\bibitem{Hardy}
G. H. Hardy, The $\int_0^{\infty} \frac{\sin x}{x}dx$, The Mathematical Gazette, Volume 5, Issue 80, July 1909, pp. 98--103.

\bibitem{Kempf}
A. Kempf, D.M. Jackson and A.H. Morales, How to (path-)integrate by differentiating,  J. Phys.: Conf. Ser. 626, 7th International Workshop DICE 2014 Spacetime - Matter - Quantum Mechanics, 012015, 2015.

\bibitem{Jia}D. Jia, E. Tang and A. Kempf, Integration by
differentiation: new proofs, methods and examples, J. Phys. A: Math.
Theor. 50, 235201, 2017.


% \bibitem{Nieto}M. M. Nieto and D. R. Truax, Phys. Lett. A 208, 8, 1995.

\bibitem{Pr1}
A.P. Prudnikov, I.U.A. Brychkov and O.I. Marichev, Integrals and Series: Elementary functions, Vol 1, Gordon and Breach Science Publishers, 1986.

\bibitem{Kempfb}
A. Kempf, D.M. Jackson and A. H. Morales, New Dirac delta function based methods with applications to perturbative expansions in quantum field theory, J. Phys. A: Math. Theor. 47, 415204, 2014. 

\bibitem{Stein}
Stein E., Harmonic Analysis: Real-Variable Methods, Orthogonality and Oscillatory Integrals, Princeton University Press, 1993.
% \bibitem{Hu}
% S. Hu, D. Kim, \& M. S. Kim (2018). 
% Jackson's integral of multiple Hurwitz–Lerch zeta functions and multiple gamma functions. Journal of Mathematical Analysis and Applications, 468(1), 227-239.

\bibitem{Kac}V. Kac, P. Cheung, Quantum calculus, Springer, 2002.

\bibitem{kurokawa}
N. Kurokawa, K. Mimachi and  M. Wakayama,
Jackson's integral of the Hurwitz zeta function, 
Rendiconti del circolo matematico di Palermo, Serie II, Tomo LVI (2007), pp. 43--56

\bibitem{poli}
Poli L. and Delerue, P., Le calcul symbolique à deux variables et ses applications, Mémorial des sciences mathématiques, vol. 127, 1954, 
\url{http://www.numdam.org/item?id=MSM_1954__127__1_0}.
\end{thebibliography}

\end{document}